\documentclass{article}
\usepackage[english]{babel}
\usepackage{amsthm,amsfonts, amsbsy, amssymb,amsmath,graphicx}
\usepackage{graphics}

\newtheorem{proposition}{Proposition}
\newtheorem{theorem}{Theorem}
\newtheorem{lemma}{Lemma}
\newtheorem{corollary}{Corollary}

\newcommand{\Z}{\mathbb Z}
\newcommand{\R}{\mathbb R}
\newcommand{\wG}{\widetilde G}

\begin{document}

%\markboth{V.O. Manturov, I.M. Nikonov}
%{On braids and groups $G_n^k$}

%%%%%%%%%%%%%%%%%%%%% Publisher's Area please ignore %%%%%%%%%%%%%%
%\catchline{}{}{}{}{}
%%%%%%%%%%%%%%%%%%%%%%%%%%%%%%%%%%%%%%%%%%%%%%%%%%%%%%%%%%%%%%%%%%%
{
\centering
{\large \bf ON BRAIDS AND GROUPS $G_n^k$}

\medskip

{Vassily Olegovich Manturov}
\footnote{
Bauman Moscow State Technical University, Russia;
Laboratory of Quantum Topology, Chelyabinsk State University, Chelyabinsk, Russia.
e-mail: vomanturov@yandex.ru}

{Igor Mikhailovich Nikonov}
\footnote{
Department of Mechanics and Mathematics, Moscow State University, Russia.
e-mail: nikonov@mech.math.msu.su}

}

\abstract{
In~\cite{M1} the first named author gave the definition of $k$-free braid groups $G_n^k$. Here we establish connections between free braid groups, classical braid groups and free groups: we describe explicitly the homomorphism from (pure) braid group to $k$-free braid groups for important cases $k=3,4$. On the other hand, we construct a homomorphism from (a subgroup of) free braid groups to free groups. The relations established would allow one to construct new invariants of braids and to define new powerful and easily calculated complexities for classical braid groups.
 }

%\keywords{braid, diagram, invariant, trisecant, graph, group}
%\ccode{20F10, 20F36, 57M25}

\section{Introduction}
In the present paper, we work with groups $G_{n}^{k}$, which generalize classical braid groups, virtual braid groups,
and other groups in a very broad sense.

Perhaps, the easiest groups are free products of cylcic groups (finite or infinite). In these groups, the word problem
and the conjugacy problem are solved extremely easily: we just contract a generator with its opposite until possible
(in a word or in a cyclic word) and stop when it is impossible.

This ``gradient descent'' algorithm allows one not only to solve the word problem but also to show that
``if a word  $w$ is irreducible then it is contained in any word $w'$ equivalent to it.'' For example,
$abcab$ is contained in $abaa^{-1}cbb^{-1}ab$ in the free product of $\Z*\Z*\Z=\langle a,b,c\rangle$.

Since the discovery of parity theory in low-dimenisonal topology by the first named author~\cite{M2},
similar phenomena arose in low-dimensional topology:

``If a diagram $K$ is odd and irreducible, then it realizes itself as a subdiagram of any other diagram $K'$ equivalent to it''~\cite{M2}.

Here ``irrreducibility'' is similar to group-theoretic irreducibility and ``oddness'' means that all crossings of the diagram are
odd or non-trivial in some sense.

In fact, parity theory allows one to endow each crossing of a diagram with a powerful diagram-like information,
so that if two crossings (resp., two letters $a$ and $a^{-1}$) are contracted then they have the same pictures.
Thus, a crossing possessing a non-trivial picture takes responsibility for the non-triviality of the whole diagram
(braid, word).

How are these two simple phenomena related to each other?

In the present paper, we study classical braid groups. However, unlike virtual braids and virtual knots,
classical braids and classical knots do not possess any good parity to realize the above principle: for
the usual Gaussian parity all crossings turn out to be even and for the component parity there are odd
crossings, but they do not lead to non-trivial invariants~\cite[Corollary 4.3]{IMN}.

However, we can change the point of view of what we call ``a crossing''. There is a natural presentation of the
braid group where generators correspond to horizontal trisecants \cite{M1}.

Besides, there is a similar presentation where generators correspond to horizontal planes having four points lying on
the same circle.

These two approaches first discovered in \cite{M1} lead to the groups $G_{n}^{3}$ and $G_{n}^{4}$.

The groups $G_{n}^{3}$ and $G_{n}^{4}$ share many nice properties with virtual braid groups and free groups.
In particular, for these groups, each crossing contains very powerful information.

In Section~\ref{sect:hom_free_group}, we coarsen this information and restrict ourselves to much coarser invariants of classical
braids which appear as the image of the map from groups $G_{n}^{k}$ to the free groups.

One immediate advantage of this approach is that we can give some obvious estimates for various complexities
of the braid which are easy to calculate. For example,

\begin{enumerate}
\item The number of generators is estimated from below by the number of non-trivial letters in the word which appears
as the image of the $G_{n}^{3}$ element in the corresponding free group.

\item The unknotting number of a braid is estimated by looking at a unique representative of a braid written according
to the presentation given in Section~\ref{sect:unknotting_number}.
\end{enumerate}

The paper is oragnized as follows. We describe homomorphisms from classical pure braids into free braid groups $G_n^3$ and $G_n^4$ in Sections 2 and 3 respectively. Section 4 is devoted to the construction of a homomorphism from even part of free braid groups $G^k_n$ into free products of $\Z_2$. In Section 5 we discuss how the defined homomorphisms can be applied to estimates of braids' complexity.

\section{Homomorphism of pure braids into $G^3_n$}\label{sect:hom_G3n}

We recall~\cite{M1} that the {\em $k$-free braid group with $n$ strands} $G_{n}^k$ is generated by elements $a_m$, $m$ is a $k$-element subset of $\{1,2,\dots,n\}$, and relations $a_m^2=1$ for each $m$, $a_m a_{m'}=a_{m'}a_m $ for any $k$-element subsets $m,m'$ such that $Card(m\cap m')\le k-2$, and {\em tetrahedron relation} $(a_{m_1}a_{m_2}\dots a_{m_{k+1}})^2=1$ for each $(k+1)$-element subset $M=\{ i_1,i_2,\dots, i_{k+1}\}\subset \{1,2,\dots,n\}$ where $m_l=M\setminus\{i_l\}, l=1,\dots,k+1$.

In~\cite{M1}, these groups appear naturally as groups describing dynamical systems of $n$ particles in some ``general position''; generators of $G_n^k$ correspond to codimension $1$ degeneracy, and relations corresponds to codimension $2$ degeneracy which occurs when performing some generic transformation between two general position dynamical systems. Dynamical system leading to $G_n^3$ and $G_n^4$ are described in~\cite{M1}; here we describe the corresponding homomorphisms explicitly.

Generators  of $G_n^3$ correspond to configurations in the evolution of dynamical systems which contain a {\em trisecant}, i.e. three particles lying in one line. Generators  of $G_n^4$ correspond to configurations which includes four particles that lie in one circle.

Let $PB_n$ be the pure $n$-strand braid group. It can be
presented with the set of generators $b_{ij}, 1\le i<j \le n,$ and
the set of relations~\cite{Bard}
\begin{align}
b_{ij}b_{kl}=b_{kl}b_{ij},\quad  i<j<k<l \mbox{ or } i<k<l<j,\\
b_{ij}b_{ik}b_{jk} = b_{ik}b_{jk}b_{ij} = b_{jk}b_{ij}b_{ik}, \quad i<j<k,\\
b_{jl}b_{kl}b_{ik}b_{jk}=b_{jl}b_{kl}b_{ik}b_{jk},\quad i<j<k<l.
\end{align}

It is well known that
\begin{proposition}
The center $Z(PB_n)$ of the group $PB_n$ is isomorphic to $\mathbb
Z$.
\end{proposition}
%Let $\theta_n$ be a generator of $Z(PB_n)$.

For each different indices $i,j$, $1\le i,j\le n$, we consider the element $c_{i,j}$ in the group $G_{n}^3$ to be the product
$$c_{i,j}=\prod_{k=j+1}^n a_{i,j,k}\cdot \prod_{k=1}^{j-1} a_{i,j,k}.$$

\begin{proposition}
The correspondence
$$b_{ij}\mapsto c_{i,i+1}^{-1}\dots c_{i,j-1}^{-1}c_{i,j}^2 c_{i,j-1}\dots c_{i,i+1}, i<j,$$
defines a homomorphism $\phi_n\colon PB_n\to G_n^3$.
\end{proposition}

\begin{proof}
Consider the configuration of $n$ points $z_k=e^{2\pi ik/n}, k=1,\dots,n$, in the plane $\mathbb R^2=\mathbb C$ where the points lie on the same circle $C=\{z\in\mathbb C\,|\,|z|=1\}$.
Pure braids can be considered as dynamical systems whose initial and final states coincide.
We can assume that the initial state coincide with the configuration considered above.
Then by Theorem 2 from~\cite{M1} there is a homomorphism $\phi_n\colon PB_n\to G_{n}^3$
and we need only describe explicitly the images of the generators of the group $PB_n$.

For any $i<j$ the pure braid $b_{ij}$ can be presented as the following dynamical system:
\begin{enumerate}
\item the point $i$ moves along the inner side of the circle $C$, passes point $i+1, i+2,\dots, j-1$ and land on the circle before the point $j$ (Fig.~\ref{fig:bij_moves} upper left);
\item the point $j$ moves over the point $i$ (Fig.~\ref{fig:bij_moves} upper right);
\item the point $i$ returns to its initial position over the points $j,j-1,\dots, i+1$ (Fig.~\ref{fig:bij_moves} lower left);
\item the point $j$ returns to its position (Fig.~\ref{fig:bij_moves} lower right).
\end{enumerate}

 \begin{figure}
  \centering
  \begin{tabular}{cc}
    \includegraphics[width=0.3\textwidth]{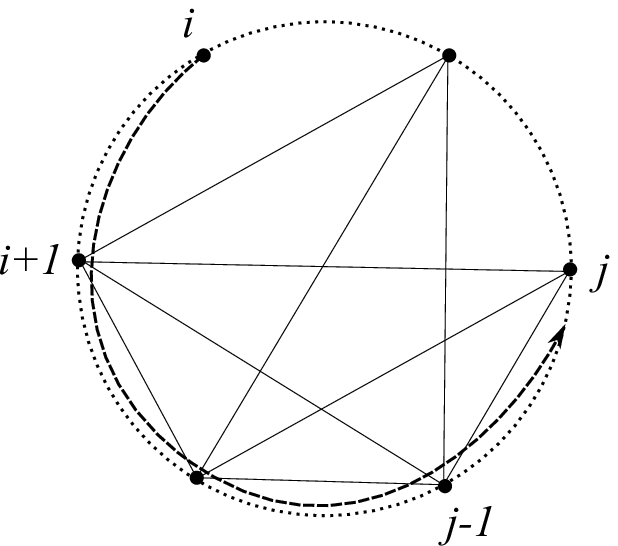} &
    \includegraphics[width=0.3\textwidth]{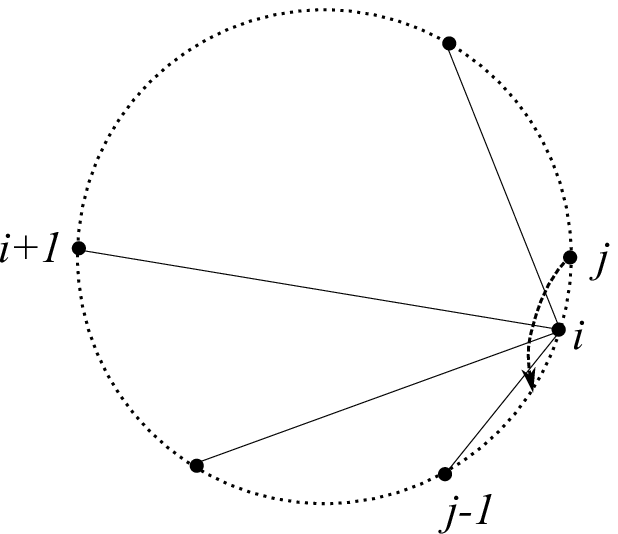} \\
    \includegraphics[width=0.3\textwidth]{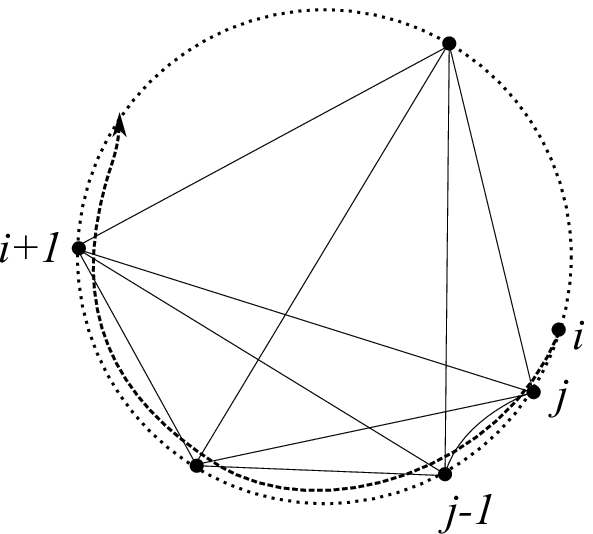} &
    \includegraphics[width=0.3\textwidth]{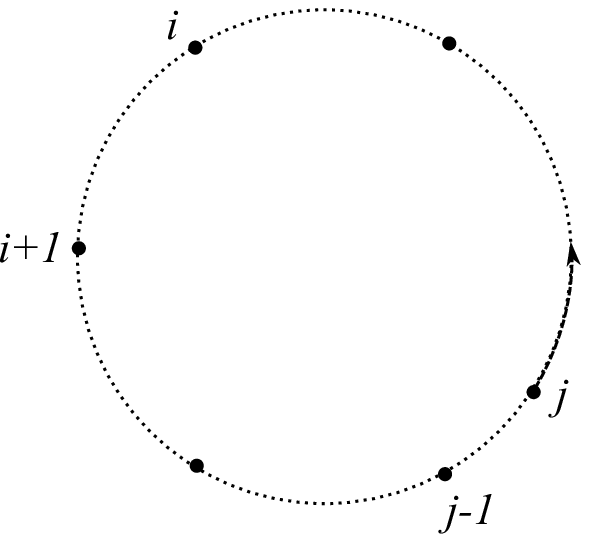}
  \end{tabular}
  \caption{Dynamical system corresponding to $b_{ij}$}\label{fig:bij_moves}
 \end{figure}

As we check all the situations in the dynamical systems where three points lie on the same line,
and write down these situations as letters in a word of the group $G_n^3$ we get exactly the element 
$$c_{i,i+1}^{-1}\dots c_{i,j-1}^{-1}c_{i,j}^2 c_{i,j-1}\dots c_{i,i+1}.$$
\end{proof}

Let $\widetilde G=\Z_2 * \Z_2 * \Z_2$ and let $a_1, a_2, a_3$ be
the generators of $\wG$. Let $\wG_{even}$ be the subgroup in $\wG$ which consists of words of even length. By obvious reason, the $2\pi$ rotation of the whole set of points around the origin has meets no trisecants, thus, the map $PB_3\to G_n^3$ has an obvious kernel.

Let us prove that in the case of $3$ strings it is the only kernel of our map.

\begin{theorem}
There is an isomorphism $PB_3/Z(PB_3)\to \wG_{even}$
\end{theorem}

\begin{proof}
The quotient group $PB_3\to PB_3/Z(PB_3)$ can be identified with a subgroup $H$ in $PB_3$ that consists of the braids for which the strands $1$ and $2$ are not linked. In other words, the subgroup $H$ is the kernel of the homomorphism $PB_3\to PB_2$ that removes the last strand. For any braid in $H$ we can straighten its strands $1$ and $2$. Looking at such a braid as a dynamical system we shall see a family of states where the particles $1$ and $2$ are fixed and the particle $3$ moves.

The points $1$ and $2$ split the line they lie on into three intervals. Let us denote the unbounded interval with the end $1$ as $a_1$, the unbounded interval with the end $2$ as $a_2$ and the interval between $1$ and $3$ as $a_3$ (see Fig.~\ref{fig:PB3}).

 \begin{figure}
  \centering
    \includegraphics[width=0.5\textwidth]{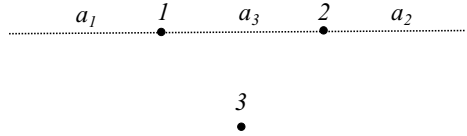}
  \caption{Initial state for a pure braid with $3$ strands}\label{fig:PB3}
 \end{figure}

We assign a word in letters $a_1,a_2,a_3$ to any motion of the point $3$ as follows. We start with the empty word. Every time the particle $3$ crosses the line $12$ we append the letter which corresponds to the interval the point $3$ crosses. If the point $3$ returns to its initial position then the word will have even number of letters. This construction defines a homomorphism $\phi\colon H\to \wG_{even}$.

On the other hand, if we have an even word in letters $a_1,a_2,a_3$ then we can define a motion of the point $3$ up to isotopy in $\mathbb R^2\setminus \{1,2\}$. Thus, we have a well defined map from even words to $H$ that induces a homomorphism $\psi\colon \wG_{even}\to H\simeq PB_3/Z(PB_3)$. It is easy to see that the homomorphisms $\phi$ and $\psi$ are mutually inverse.
\end{proof}

The crucial observation which allows us to prove that this map is an isomorphism, is that we can restore the dynamics from the word in the case of $3$ strands. Since relations in  $G_n^3$ correspond to relations in the braid group, this suffices to prove the isomorphism.

When $n>4$, not all words in $G^3_n$ correspond to the dynamical systems. thus, the question whether $(PB_n)/Z(PB_n)\to G^3_n$ requires additional techniques.

Actually, by using similar methods, one can construct monomorphic map $PB_n\to G^3_{n+1}$; this will be discussed in a separate paper.

\section{Homomorphism of pure braids into $G^4_n$}\label{sect:hom_G4n}

In the present section, we describe an analogous mapping $PB_n\to G^4_n$; here points $z_1,z_2,\dots,z_n$ on the plane are in general position of no four points of them belong to the same circle (or line); codimension $1$ degeneracies will correspond to generators of $G^4_n$, where at some moment exactly one quadruple of points belongs to the same circle (line), and relations correspond to the case of more complicated singularities.

Let $a_{\{i,j,k,l\}}, 1\le i,j,k,l\le n$ be the generators of the group $G^4_n, n>4$.
% We assume $a_{\{i,j,k,l\}}=1$ when there are two equal indices among $i,j,k,l$.

Let $1\le i<j\le n$. Consider the elements
\begin{eqnarray}\label{eq:d_elements}
c^{I}_{ij}=\prod_{p=2}^{j-1}\prod_{q=1}^{p-1}a_{\{i,j,p,q\}},\\
c^{II}_{ij}=\prod_{p=1}^{j-1}\prod_{q=1}^{n-j}a_{\{i,j-p,j,j+q\}},\\
c^{III}_{ij}=\prod_{p=1}^{n-j+1}\prod_{q=0}^{n-p+1}a_{\{i,j,n-p,n-q\}},\\
c_{ij}=c^{II}_{ij} c^{I}_{ij} c^{III}_{ij}.
\end{eqnarray}

\begin{proposition}\label{prop:hom_G4n}
The correspondence

\begin{equation}\label{eq:maps_to_G4n}
b_{ij}\mapsto c_{i,i+1}\dots c_{i,j-1}c_{i,j}^2 c_{i,j-1}^{-1}\dots c_{i,i+1}^{-1},\quad i<j,
\end{equation}
defines a homomorphism $\phi_n\colon PB_n\to G_n^4$.
\end{proposition}

In order to construct this map explicitly, we have to indicate the initial state in the configuration space of $n$-tuple of points.

By obvious reason, the initial state with all points lying on the circle does not work; so, we shall use the parabola instead.

Let $\Gamma = \{ (t,t^2)\,|\, t\in\R\}\subset \R^2$ be the graph of the function $y=x^2$. Consider a rapidly increasing sequence of positive numbers $t_1,t_2,\dots, t_n$ (precise conditions on the sequence growth will be formulated below) and denote the points $(t_i,t_i^2)\in\Gamma$  by $P_i$.

Pure braids can be considered as dynamical systems whose initial and final states coincide and we assume that the initial state
is the configuration $\mathcal P=\{P_1,P_2,\dots,P_n\}$. Then by Theorem~2 from~\cite{M1} there is a homomorphism $\psi_n\colon PB_n\to G_{n}^4$.
We need to describe explicitly the images of the generators of the group $PB_n$.

For any $i<j$ the pure braid $b_{ij}$ can be presented as the following dynamical system: the point $i$ moves along the graph $\Gamma$
\begin{enumerate}
\item the point $i$ moves along the graphics $\Gamma$ and passes points $$P_{i+1}, P_{i+2},\dots, P_j$$ from above (Fig.~\ref{fig:G4_bij_moves} upper left);
\item the point $j$ moves from above the point $i$ (Fig.~\ref{fig:G4_bij_moves} upper right);
\item the point $i$ moves to its initial position from above the points $$P_{j-1},\dots, P_{i+1}$$ (Fig.~\ref{fig:G4_bij_moves} lower left);
\item the point $j$ returns to its position (Fig.~\ref{fig:G4_bij_moves} lower right).
\end{enumerate}

 \begin{figure}
  \centering
  \begin{tabular}{cc}
    \includegraphics[width=0.25\textwidth]{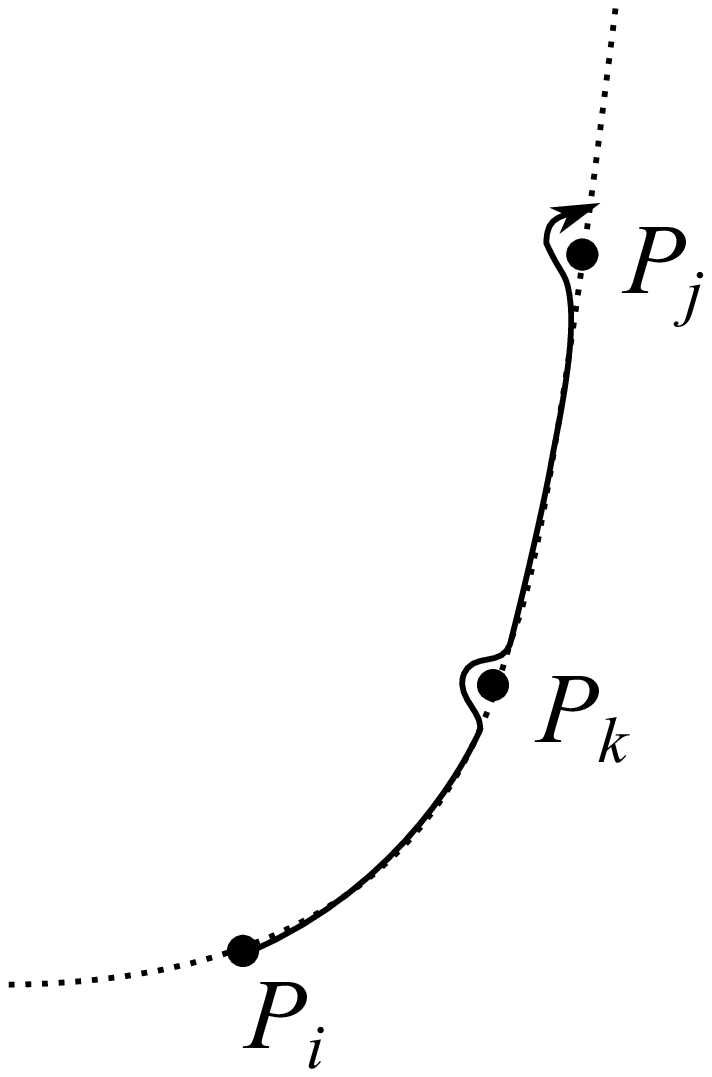} &
    \includegraphics[width=0.25\textwidth]{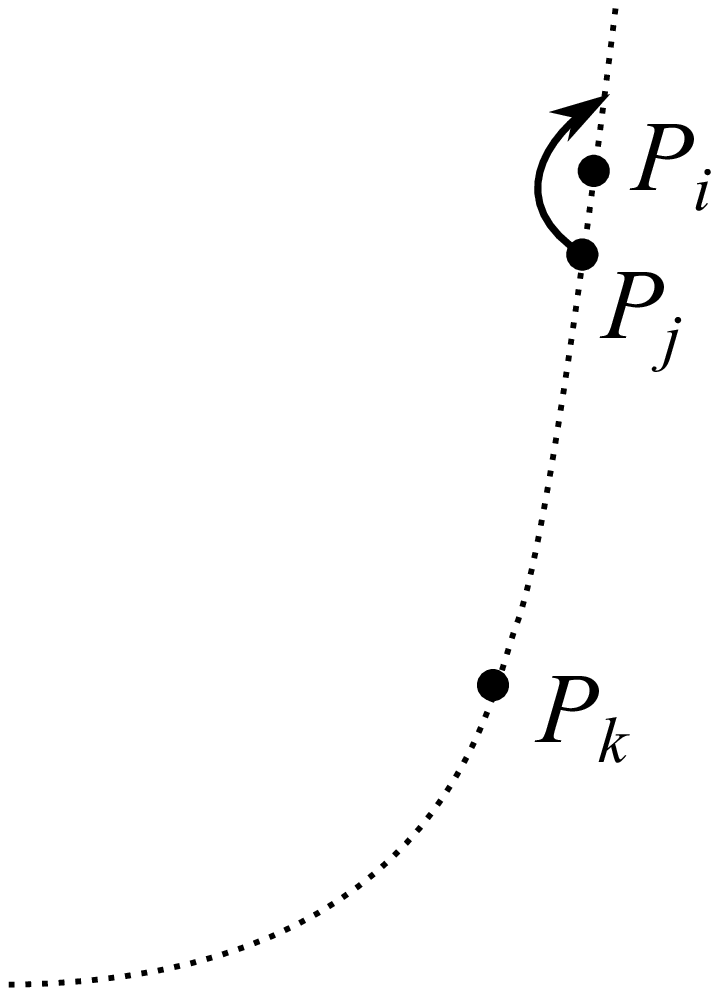} \\
    \includegraphics[width=0.25\textwidth]{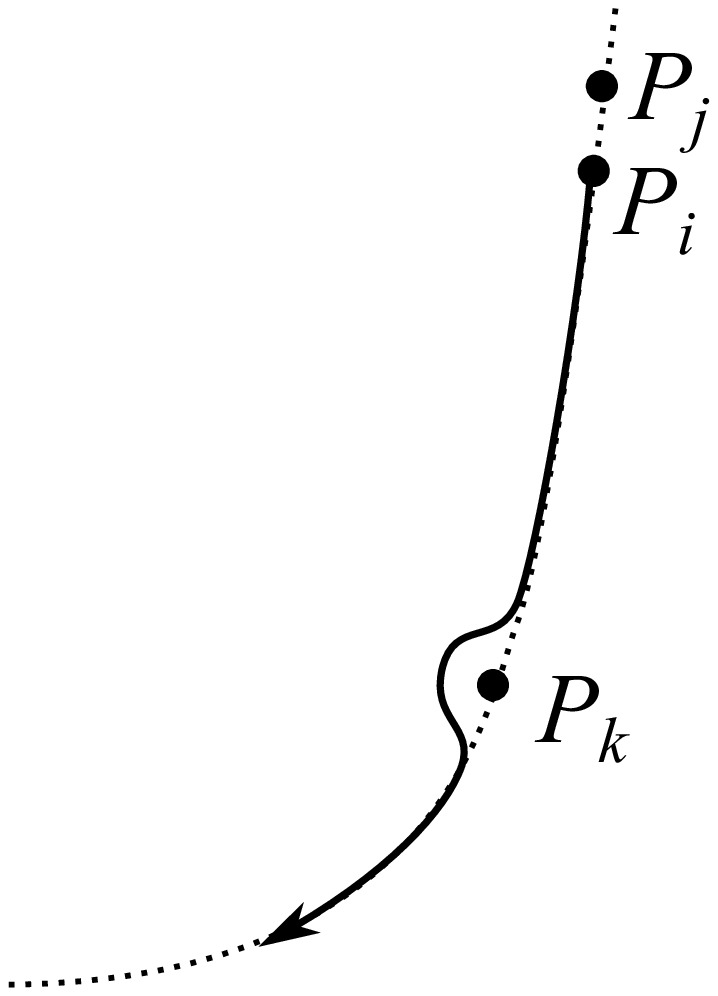} &
    \includegraphics[width=0.25\textwidth]{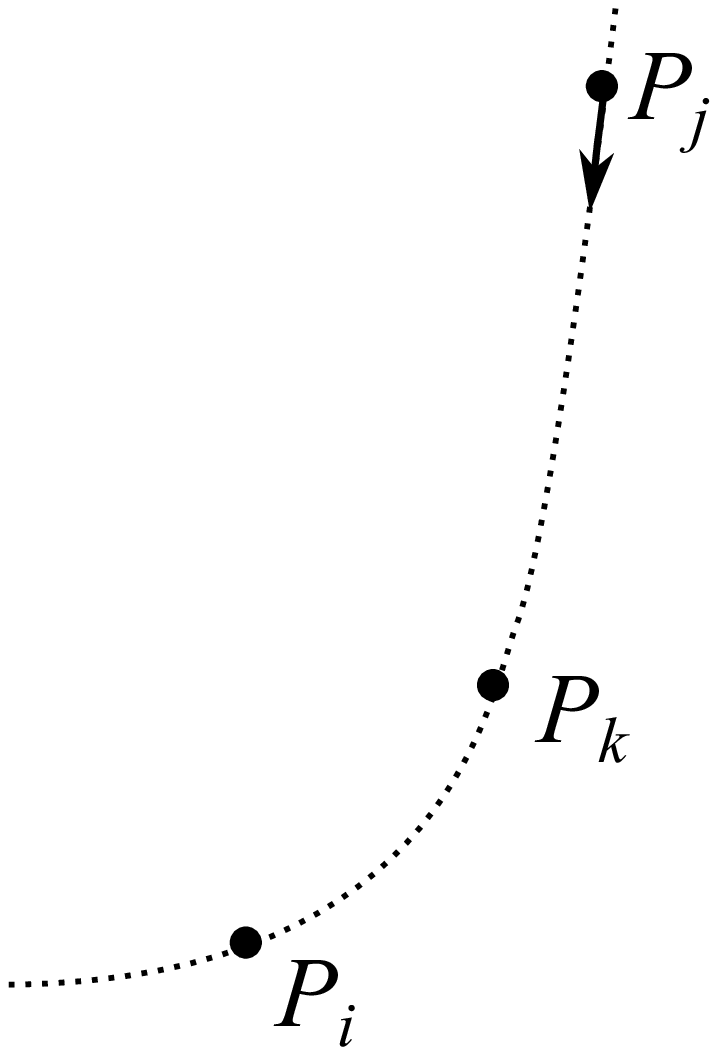}
  \end{tabular}
  \caption{Dynamical system which corresponds to $b_{ij}$}\label{fig:G4_bij_moves}
 \end{figure}

The proof that this dynamical system leads to the presentation~\eqref{eq:maps_to_G4n} is given in Appendix.

\section{Homomorphism into a free group}\label{sect:hom_free_group}

Let $\mathcal N=\{1,2,\dots,n\}$ be the set of indices.

Let $H^k_n\subset G^k_n$ be the subgroup whose elements are given by the {\em even} words, that is words which include any generator $a_m, m\subset\mathcal N, |m|=k$, evenly many times.
We construct a homomorphism $\phi$ from the subgroup $H^k_n$ to the free product of $2^{(k-1)(n-k)}$ copies of the group $\mathbb Z_2$.

Roughly speaking, any letter $a_m$ in $G_n^k$ will get a collection of ``colours'' coming from the interaction of indices from $m$ with other indices. These ``colours'' will remain unchanged when performing the relations $a_ma_{m'}=a_{m'}a_m$ and the tetrahedron relation $\dots a_m\dots = \dots a_m\dots$; moreover, the two  adjacent letters $a_ma_m$ which contracts to $1$ inside the word will have the same colour.

Let $a_m, m\subset\mathcal N, |m|=k$, be a generator of the group $G^k_n$. Without loss of generality we can suppose that $m=\{1,2,\dots,k\}$.
Assume that $p\in\mathcal N\setminus m$. For any $1\le i\le k$ consider the set $m[i]=m\setminus\{i\}\cup\{p\}$.
Let us define a homomorphism $\psi_p\colon G^k_n\to \mathbb Z_2^{\oplus k-1}$ by the formulas
$\psi_p(a_{m[i]})=e_i, 1\le i\le k-1, \psi_p(a_{m[k]})=\sum_{i=1}^{k-1}e_i$ and $\psi_p(a_{m'})=0$ for the other $m'\in\mathcal N, |m'|=k$.
Here $e_1,\dots, e_{k-1}$ denote the basis elements of the group $\mathbb Z_2^{\oplus k-1}$.
The homomorphism $\psi_p$ is well defined since the relations in $G^k_n$ are generated by even words.

Consider the homomorphism $\psi = \bigoplus_{p\not\in m}\psi_p$, $\psi\colon G^k_n\to Z$, where $Z=\mathbb Z_2^{\oplus (k-1)(n-k)}$.
Note that $H^k_n\subset\ker\psi$.

Let $H=\mathbb Z_2^{\ast Z}$ be the free product of $2^{(k-1)(n-k)}$ copies of the group $\mathbb Z_2$
and the exponents of its elements are indiced with the set $Z$. Let $f_x, x\in Z,$ be the generators of the group $H$.
Define the action of the group $G^k_n$ on the set $Z\times H$ by the formula
$$a_m'\cdot (x,y)=\left\{\begin{array}{cl}(x,f_xy) & m'=m, \\ (x+\psi(a_{m'}),y) & m'\neq m.\end{array}\right.
$$

This action is well defined. Indeed, $a_{m'}^2\cdot (x,y)=(x,y)$ for any $m'\in\mathcal N, |m'|=k$. On the other hand,
for any tuple $\tilde m=(i_1,i_2,\dots, i_{k+1})$ such that $m\not\subset \tilde m$
one has $\left(\prod_{l=1}^{k+1}a_{\tilde m\setminus\{i_l\}}\right)^2\cdot (x,y)=(x,y)$ since the generators act trivially here.
If $m\subset \tilde m$ then $\left(\prod_{l=1}^{k+1}a_{\tilde m\setminus\{i_l\}}\right)^2=w_1 a_m w_2 a_m w_3$, where the word $w_2$
and the word $w_1w_3$ are products of generators $m[i], i=1,\dots,k$ and
\begin{multline*}
\left(\prod_{l=1}^{k+1}a_{\tilde m\setminus\{i_l\}}\right)^2\cdot (x,y)=w_1 a_m w_2 a_m w_3 \cdot (x,y) =\\
 w_1 a_m w_2 a_m \cdot (x+\psi(w_3),y) =
  w_1 a_m w_2 \cdot (x+\psi(w_3),f_{x+\psi(w_3)}y)=\\ w_1 a_m\cdot (x+\psi(w_3),f_{x+\psi(w_3)}y)= 
  w_1\cdot(x+\psi(w_3),f_{x+\psi(w_3)}^2y) =\\ w_1\cdot(x+\psi(w_3),y) = (x+\psi(w_1)+\psi(w_3),y)=(x,y).
\end{multline*}
The fourth and the last equalities follow from the equality $\psi(w_1)+\psi(w_3)=\psi(w_1w_3)=\psi(w_2)=\sum_{i=1}^k\psi(a_{m[i]})=0$.

For any element $g\in G^k_n$ define $\psi_x(g)$ from the relation $$g\cdot (x,1)=(x+\psi(g),\psi_x(g)1)$$ and let $\phi(g)=\phi_0(g)$.
Then $g\cdot (x,y)=(x+\psi(g),\psi_x(g)y)$ for any $(x,y)\in Z\times H$.
If $g\in H^k_n$ then  $\psi(g)=0$ and $g\cdot (0,y)=(0,\psi(g)y)$. Hence, for any $g_1,g_2\in H^k_n$
on has $(g_1g_2)\cdot(0,1)= g_1\cdot(g_2\cdot(0,1))=g_1\cdot(0,\phi(g_2)) = (0,\phi(g_1)\phi(g_2))$.
On the other hand, $(g_1g_2)\cdot(0,1)=(0,\phi(g_1g_2))$. Thus, $\phi\colon H^k_n\to H$ is a homomorphism.

For any element $x$ in $H$, let $c(x)$ denote the complexity of $x$, i.e. the length of the irreducible representative word of $x$.

Let $b\in PB_n$ be a classical braid and $\beta = \phi_n(b)$ be the corresponding free braid in $G^3_n$. Note that the word $\beta$ belongs to $H^3_n$. Then for any $m\subset \mathcal N, |m|=3$, the map $\phi$ defined above is applicable to $\beta$. The geometrical complexity of the braid $b$ can be estimated by complexity of the element $\phi(\beta)$.

\begin{proposition}
The number of horizontal trisecants of the braid $b$ is not less than $c(\phi(\beta))$.
\end{proposition}
The proof of the statement follows from the definition of the maps $\phi_n$ and $\phi$.

Using the homomorphism $PB_n\to G^4_n$, we get an analogous estimation for the number of ``circled quadrisecants'' of the braid.

\section{Free Groups and Crossing Numbers}\label{sect:unknotting_number}

The homomorphism described above allows one to estimate the new complexity for the braid group $BP_n$ by using $G^3_n$ and $G^4_n$. These complexities have an obvious geometrical meaning as the estimates of the number of horizontal trisecants ($G^3_n$) and   ``circled quadrisecants'' ($G^4_n$).

Let $G=\Z_{2}^{*3}$ be the free product of three copies of $\Z_{2}$ with generators $a,b,c$ respectively.
A typical example of the word in this group is

\begin{equation}\label{eq:word_in_free_group}
w=abcbabca...
\end{equation}

 \begin{figure}
  \centering
    \includegraphics[width=0.25\textwidth]{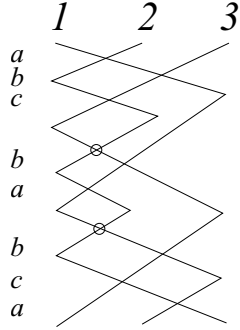}
  \caption{Braid which corresponds to the word $w$}\label{fig:braid_of_word}
 \end{figure}

Note that the word~\eqref{eq:word_in_free_group} is irreducible, so, every word equivalent to it contains it as a subword.
In particular, this means that {\em every word $w'$ equivalent to $w$ contains at least $8$ letters.}

The above mentioned complexity is similar to a ``crossing number'', though, crossings are treated
in a non-canonical way.

In~\cite{M2}, it is proved that for free knots (which are knot theoretic analogs of the group $G^2_n$) if a knot $K$ is complicated enough then it is contained (as a smoothing) in any knot equivalent to it.

In~\cite{M3}, similar statements are proved for $G^2_n$.

Let us now treat one more complicated issue, the unknotting number.

To make the issue simpler, let us start with the toy model.
Let $G'=\Z^{*3}$ be the free product of the three copies of $\Z$ with generators $a,b,c$.

Assume we are allowed to perform the operation of switching the sign $a\longleftrightarrow a^{-1}$,
$b\longleftrightarrow b^{-1}$, $c\longleftrightarrow c^{-1}$ along with the usual reduction
of opposite generators $aa^{-1},bb^{-1},cc^{-1},a^{-1}a,\dots$.

Given a word $v$; how many switches do we need to get from $v$ to the word representing the trivial
element? How to estimate this number from below?

For the word $w=abcbabca$ the answer is ``infinity''. It is impossible to get from $w$ to $1$ in $G'$
because $w$ represents a non-trivial element of $G$, and the operations $a\longleftrightarrow a^{-1}$
do not change the element of $G$ written in generators $a,b,c$.

Now, let the element of $G'$ be given by the word
$w'=a^{4}b^{2}c^{4}b^{-4}$.

This word $w'$ is trivial in $G$, however, if we look at exponents of $a,b,c$, we see that they are
$+4,-2,+4$. Thus, the number of switchings is bounded from below by $\frac 12(|4|+|2|+|4|)=5$, and one can easily find how to make the word $w'$ trivial in $G'$ with five switchings.

For classical braids, a crossing switching corresponds to one turn of one string of a braid around another. If $i$ and $j$ are the numbers of two strings of some braid $B$ then the word in $G^k_n$, $k=3,4$, which corresponds to the dynamical system describing the full turn of the string $i$ around the string $j$, looks like $c_{ij}^2$, where $c_{ij}=\prod_{m\supset\{i,j\}}a_m)$ and the product is taken over the subsets $m\subset\{1,2,\dots, n\}$ which have $k$ element and contain $i$ and $j$, given in some order. Note that the word $c_{ij}^2$ is even.

Fix a $k$-element subset $m\subset\{1,2,\dots, n\}$ containing $\{i,j\}$. Consider the homomorphism $\phi\colon H^k_n\to H$ determined by $m$. The word  $c_{ij}^2$ yields the factor $f_xf_{x+z_{ij}}\in H$ where the index $x\in Z$ depends on the positition of $c_{ij}^2$ inside the word $\beta$ of $G^k_n$ that corresponds to the braid $B$, and the element $z_{ij}=\sum_{m'\,:\, m'\supset\{i,j\}, Card(m\cap m')=k-1}\psi(m')\in Z$ depends only on $i$, $j$, and $m$ and does not depend on the order in the product $c_{ij}$. Hence, an addition of a full turn of the string $i$ around the string $j$ corresponds to substitution of the subword $1=f_x^2$ with the subword $f_xf_{x+z_{ij}}$, i.e. a crossing switch corresponds to switch of elements $f_x$ and $f_{x+z_{ij}}$ in the image of the braid under the homomorphism $\phi$. Thus, the following statement holds.

\begin{proposition}\label{prop:unknotting_number_estimation}
The unknotting number of a braid $B$ is estimated from below by the number of switches $f_x\mapsto f_{x+z_{ij}}$ which are necessary to make the word $\phi(B)\in H$ trivial.
\end{proposition}

The number of switches in the Proposition~\ref{prop:unknotting_number_estimation} is not an explicit characteristic of a word in the group $H$, but one can give several rough estimates for it which can be computed straightforwardly. For example, consider the following construction. Let $\pi\colon H\to \Z_2[Z]$ be the natural projection; let $Z_0$ be the subgroup in $Z$ generated by the elements $z_{ij}, \{i,j\}\subset m$. For any element $\xi = \sum_{z\in Z}\xi_z z\in \Z_2[Z]$ and any $z\in Z$ consider the number $c_z(\xi)=Card\left(\{z_0\in Z_0\,|\, \xi_{z+z_0}\ne 0\}\right)$, and let $c(\xi)=\max_{z\in Z} c_z(\xi)$.

Let $\omega\in H$ be an arbitrary word and $\xi=\pi(\omega)$. Any switch $f_x\mapsto f_{x+z_{ij}}$ in $\omega$ corresponds to a switch of $\xi$. Note that these switches map any element $z\in Z$ into the class $z+Z_0\subset Z$. After a switch, two summands of $\xi$, that correspond to generators in $z+Z_0$, can annihilate. Thus, removing all the summands of $\xi$ by the generators in $z+Z_0$ takes at least $\frac 12 c_z(\xi)$ switches. Thus, we get the following estimate for the unknotting number.   

\begin{proposition}\label{prop:unknotting_number_estimation_rough}
The unknotting number of a braid $B$ is estimated from below by the number $\frac 12 c\left(\pi(\phi(B))\right)$.
\end{proposition}

Consider the braid $B$ in Fig.~\ref{fig:braid_example}. Its unknotting number is $2$. The braid corresponds to the element
$$\beta=a_{123}a_{234}a_{123}a_{134}a_{123}a_{134}a_{123}a_{234}\in G^3_4.$$
Let $m=\{1,2,3\}$. Then $z_{12}=\psi(a_{124})=e_1+e_2\in Z=\Z_2^{\oplus 2}$, $z_{13}=\psi(a_{134})=e_2$, $z_{23}=\psi(a_{234})=e_1$.

The image of the element $\beta$ is $\phi(\beta)=f_{0}f_{e_1}f_{e_1+e_2}f_{e_1}$. It is easily to see the word $\phi(\beta)$ can not be trivialized with one switch, but it can be made trivial with two switches:
$$
f_{0}f_{e_1}f_{e_1+e_2}f_{e_1}\stackrel{z_{13}}{\longrightarrow}f_{0}f_{e_1}f_{e_1}f_{e_1}=f_{0}f_{e_1}\stackrel{z_{23}}{\longrightarrow}f_{0}f_{e_0}=1.
$$
Note that the Proposition~\ref{prop:unknotting_number_estimation_rough} gives the estimate $$\frac 12 c(\pi(\phi(\beta)))=\frac 12 c(f_{0}+f_{e_1+e_2})=1.$$

 \begin{figure}
  \centering
    \includegraphics[width=0.25\textwidth]{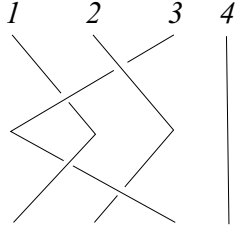}
  \caption{Braid with the unknotting number equal to $2$}\label{fig:braid_example}
 \end{figure}

\section*{Appendix}
{\scriptsize % begin of appendix

\begin{proof}(Proof of Proposition~\ref{prop:hom_G4n})

\begin{lemma}\label{lem:bifurcation_condition_on_graphics}
Let $A_i=(x_i,x^2_i), i=0,1,2,3$ be different points in the graph $\Gamma$. Then $A_0,A_1,A_2,A_3$ belong to one circle if and only if the sum $x_0+x_1+x_2+x_3=0$.
\end{lemma}

\begin{proof}
Assime that $A_0, A_1, A_2, A_3$ belong to the circle with the center $(a,b)$ and the radius $r$ for some $a,b,r$. Then the following equations hold:
$(x_i-a)^2+(x^2_i-b)^2=r^2$, $i=0,\dots,3$. Subtracting the first equation from the last three equations, we obtain a linear system on variables $a$ and $b$
with three equations $2(x_0-x_i)a+2(x_0^2-x_i^2)b+(x_i^2-x_0^2+x_i^4-x_0^4)=0$, $i=1,2,3$. The combatibility equation for this system is $\Delta=0$ where
$$
\Delta = \left|\begin{array}{ccc}
x_0-x_1 & x_0^2-x_1^2 & x_1^2-x_0^2+x_1^4-x_0^4\\
x_0-x_2 & x_0^2-x_2^2 & x_2^2-x_0^2+x_2^4-x_0^4\\
x_0-x_3 & x_0^2-x_3^2 & x_3^2-x_0^2+x_3^4-x_0^4
\end{array}\right|
$$
is the determinant of the system. But

\begin{multline*}
\Delta = \left|\begin{array}{ccc}
x_0-x_1 & x_0^2-x_1^2 & x_1^4-x_0^4\\
x_0-x_2 & x_0^2-x_2^2 & x_2^4-x_0^4\\
x_0-x_3 & x_0^2-x_3^2 & x_3^4-x_0^4
\end{array}\right| =\\ (x_0-x_1)(x_0-x_2)(x_0-x_3)(x_1-x_2)(x_1-x_3)(x_2-x_3)(x_0+x_1+x_2+x_3).
\end{multline*}
Hence, the condition $\Delta=0$ is equivalent to the equality $x_0+x_1+x_2+x_3=0$.
\end{proof}

\begin{corollary}\label{cor:no_additional_points}
Let $\Gamma_+=\{ (t,t^2)\,|\, t\ge 0\}$ be the positive half of the graphics $\Gamma$ and $X=\{P_1,P_2,\dots,P_n\}\subset\Gamma_+, 4\le n,$ be a finite subset.
Denote the circle which contains $3$ different points $P_i,P_j,P_k$ of $X$ by $C_{ijk}$. Then
$$\Gamma_+\cap\bigcup_{1\le i<j<k\le n}C_{ijk} = X.$$
\end{corollary}

\begin{proof}
For any $i<j<k$ one has $C_{ijk}\cap\Gamma = \{P, P_i,P_j,P_k\}$ where $P=(t,t^2)$, $P_i=(t_i,t_i^2)$, $P_j=(t_j,t_j^2)$, $P_k=(t_k,t_k^2)$. Since $t_i,t_j,t_k\ge 0$ and
$t+t_i+t_j+t_k=0$, then $t<0$ and $P\not\in\Gamma_+$. Hence, $C_{ijk}\cap\Gamma_+ = \{P_i,P_j,P_k\}$.
\end{proof}

The corollary means that there is no additional intersections of the circles generated by the points of $X$ with $\Gamma_+$. Then in the dynamical systems which corresponds to the pure braid $b_{ij}$, a configuration of four points lying on one circle can occur only when point $i$ (or $j$) leaves the curve $\Gamma_+$. If that configuration appears when the point $i$ moves above the point $P_k, i<k\le j,$ it means that point $i$, $P_k$ and some other two points $P_l$ and $P_m$  lie on the same circle, i.e. point $i$ belongs to the circle $C_{klm}$. We assume below that $t_l<t_m$.

We need to find the order the point $i$ intersects circles $C_{klm}$ when moving near $P_k$. We can substitute circles with their tangent lines at $P_k$ under assumption the point $i$ moves close to $P_k$.

Given points $P_k,P_l,P_m\in\Gamma_+$ there are three possible cases:
\begin{enumerate}
\item $t_l<t_m<t_k$;
\item $t_l<t_k<t_m$;
\item $t_k<t_l<t_m$.
\end{enumerate}

 \begin{figure}
  \centering

    \includegraphics[width=0.25\textwidth]{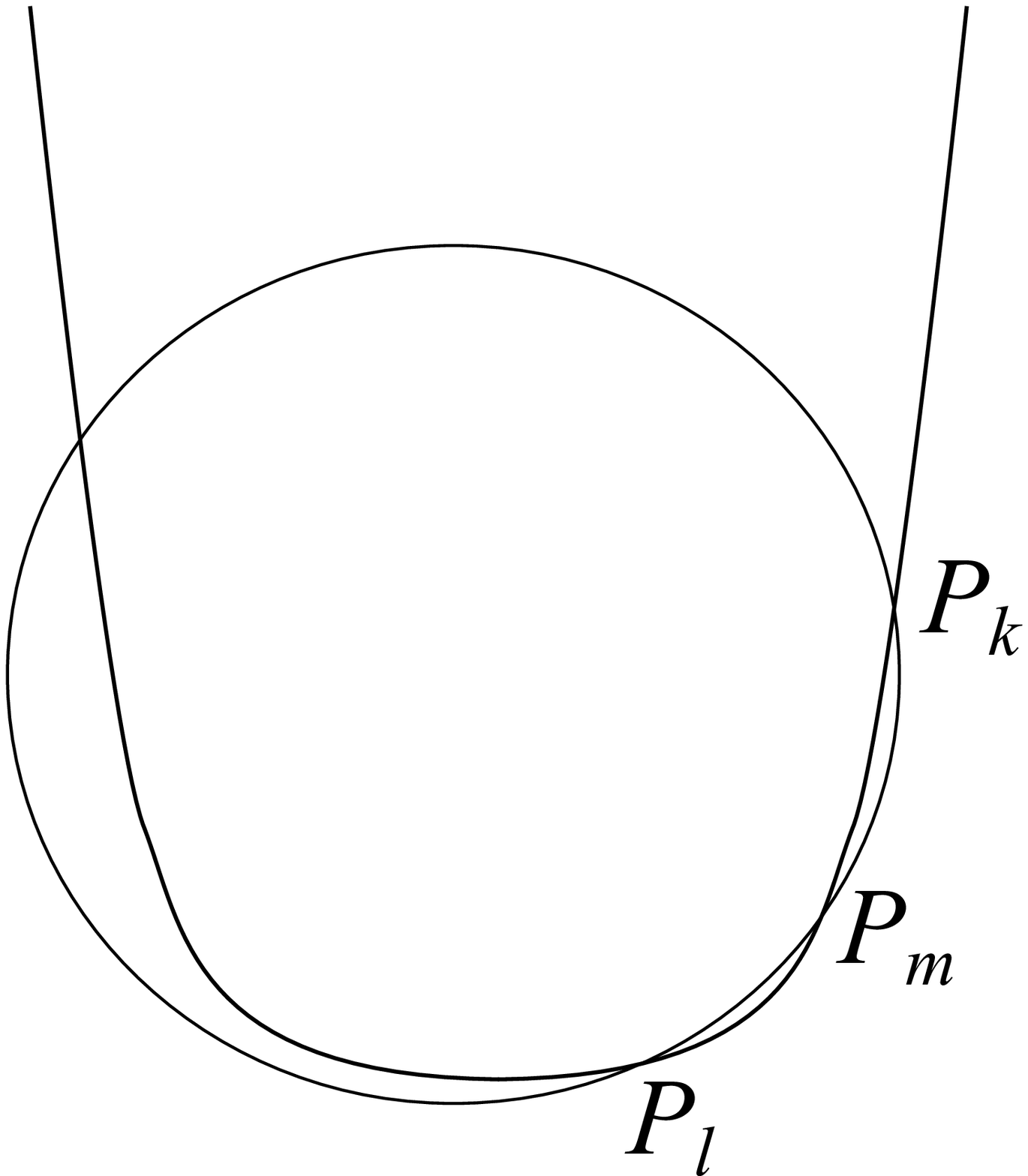} \quad
    \includegraphics[width=0.25\textwidth]{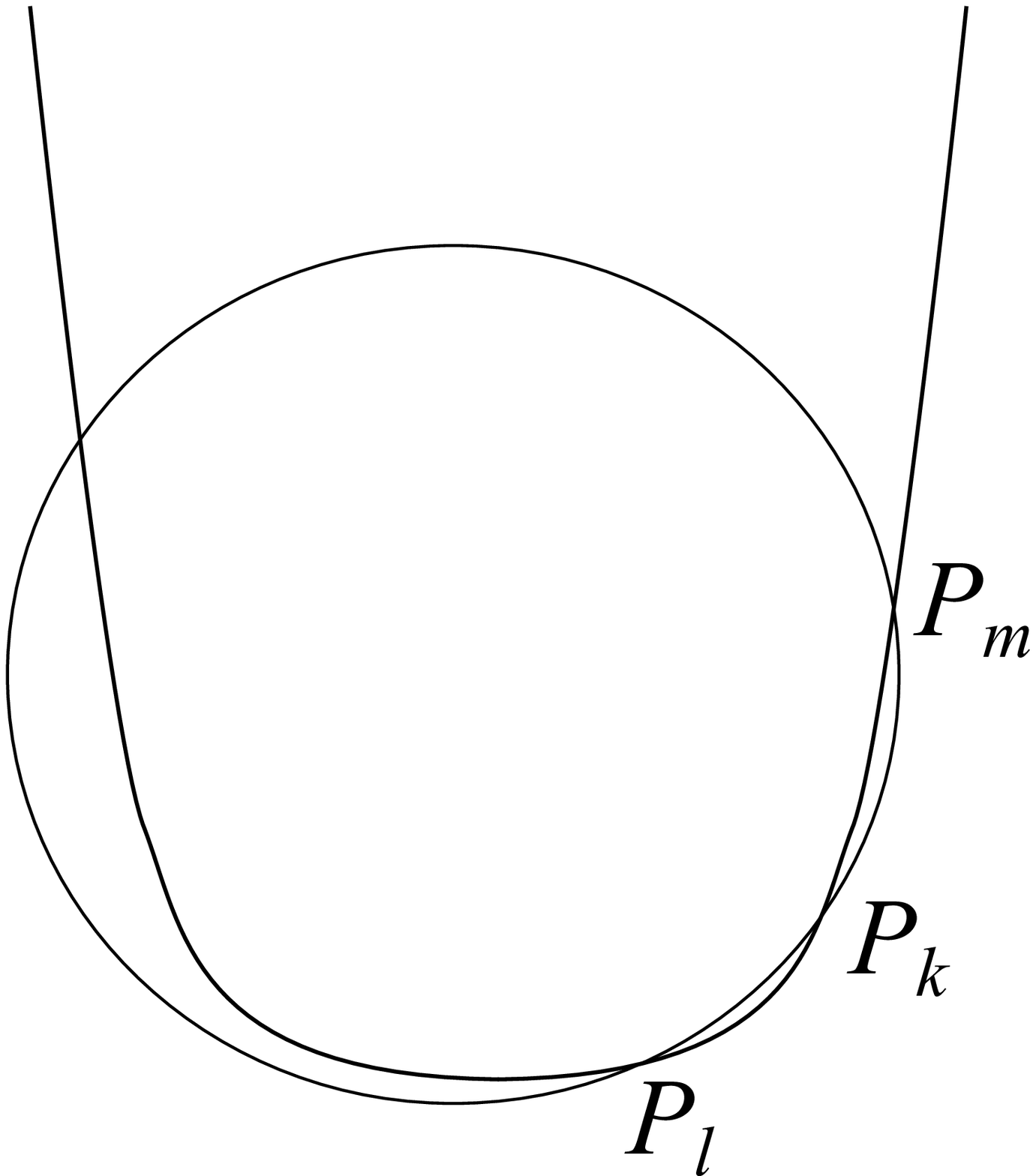} \quad
    \includegraphics[width=0.25\textwidth]{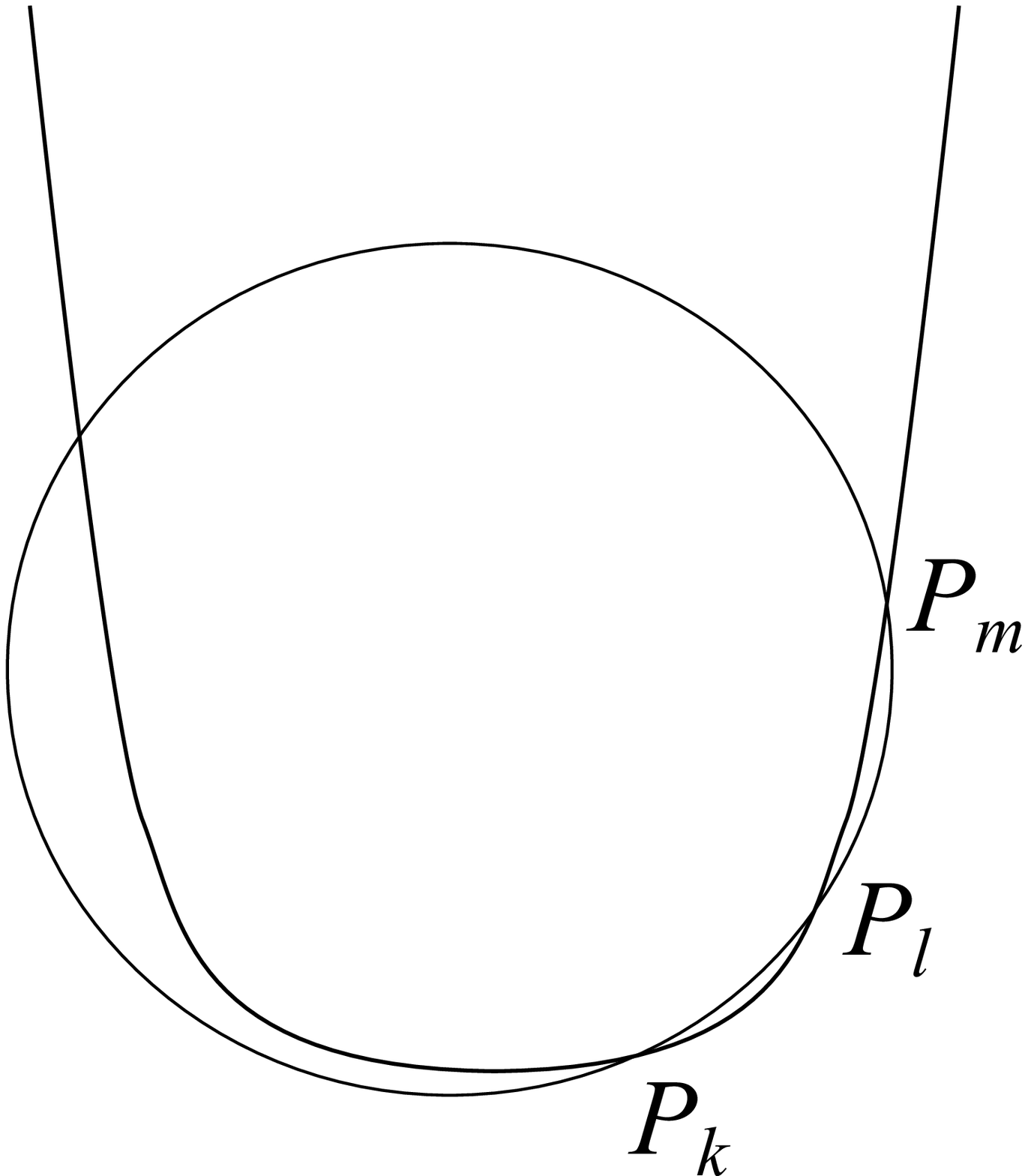}

  \caption{Dynamic system which corresponds to $b_{ij}$}\label{fig:circle_cases}
 \end{figure}

%Let $Q$ be the fourth intersection
Case $t_l<t_m<t_k$. The following statement can be obtained by straightforward computation.

\begin{lemma}\label{lem:slope_formula}
The slope of the tangent line of the circle $C_{klm}$, which contains the points $P_k=(t_k,t_k^2)$, $P_l=(t_l,t_l^2)$ and $P_m=(t_m,t_m^2)$, is equal to
\begin{equation}\label{eq:slope_formula}
\kappa_{k,lm}=-\frac{t_k^2(t_l+t_m)+t_k((t_l+t_m)^2+2)+t_l t_m(t_l+t_m)}{t_k^2-t_k(t_l+t_m)-(t_l^2+t_l t_m+t_m^2+1)}
\end{equation}
at the point $P_k$.
\end{lemma}

If the sequence $t_i, 1\le i\le n,$ grows rapidly then the coefficient $\kappa_{k,lm}$ is close to the sum $-(t_l+t_m)$. More precisely, we have the following estimate.

\begin{lemma}\label{lem:slope_estimate}
Let $t_1\ge 1$ and
\begin{equation}\label{eq:growth_condition_case1}
t_i\ge 100 t_{i-1}^2
\end{equation}
 for any $i, 1<i\le n$. Then for any $1\le l<m<k\le n$ we have
$$-(t_l+t_m+1)<\kappa_{k,lm}<-(t_l+t_m).
$$
\end{lemma}
\begin{proof}
The numerator in the formula~\eqref{eq:slope_formula} is greater than $t_k^2(t_l+t_m)$ and the denominator is less than $t_k^2$. Hence, $\kappa_{k,lm}<-(t_l+t_m)$. On the other hand,
$$
t_k^2-t_k(t_l+t_m)-(t_l^2+t_l t_m+t_m^2+1)> t_k^2-6 t_k t_m\ge t_k^2 \left( 1-\frac 6{100 t_m}\right),
$$
and since $\left( 1-\frac 6{100 t_m}\right)\left( 1+\frac 7{100 t_m}\right)=1+\frac 1{100 t_m}-\frac {42}{10000 t_m^2}>1$, we have $$\left( 1-\frac 6{100 t_m}\right)^{-1}<1+\frac 7{100 t_m}.$$ Moreover,
$$ t_k^2(t_l+t_m)+t_k((t_l+t_m)^2+2)+t_l t_m(t_l+t_m)<t^2_k \left(t_l+t_m+\frac {7t_m^2}{t_k}\right)<t^2_k \left(t_l+t_m+\frac {7}{100}\right)
$$
since $(t_l+t_m)^2+2< (2t_m)^2+t_m^2=5t_m^2$ and $t_l t_m(t_l+t_m)<2t_kt_m^2$. Then
\begin{multline*}
-\kappa_{k,lm}<\frac{t^2_k \left(t_l+t_m+\frac {7}{100}\right)}{t_k^2 \left( 1-\frac 6{100 t_m}\right)}<\left(t_l+t_m+\frac {7}{100}\right)\left(1+\frac 7{100 t_m}\right)=\\ t_l+t_m+\frac {7t_l+14t_m}{100 t_m}+\frac {49}{10000 t_m}<t_l+t_m+\frac {21}{100}+\frac {49}{10000}<t_l+t_m+1
\end{multline*}
\end{proof}

\begin{corollary}\label{cor:case1_order}
Let $t_1\ge 1$ and $t_i\ge 100 t_{i-1}^2$ for any $i, 1<i\le n$. Then for any indices $k,l_1,l_2,m_1,m_2$ such that $1\le l_1<m_1<k\le n$ and $1\le l_2<m_2<k$ the inequality $\kappa_{k,l_1m_1}>\kappa_{k,l_2m_2}$ holds if and only if $m_1<m_2$ or $m_1=m_2, l_1<l_2$.
\end{corollary}
\begin{proof}
If  $m_1<m_2$ then
$$\kappa_{k,l_2m_2}<-(t_{l_2}+t_{m_2})<-100t_{m_1}^2<-(2t_{m_1}+1)<\kappa_{k,l_1m_1}.$$
If $m_1=m_2$ and $l_1<l_2$ then
$$\kappa_{k,l_2m_2}<-(t_{l_2}+t_{m_2})<-(100t_{l_1}^2+t_{m_1})<-(t_{l_1}+t_{m_1}+1)<\kappa_{k,l_1m_1}.$$
\end{proof}

Corollary~\ref{cor:case1_order} determines the order the point $i$ intersects the circles $C_{klm}$. Since $\kappa_{k,lm}<0$ the tangent line to $C_{klm}$ that lies above the graphics $\Gamma_+$, belongs to the second quadrant of the plane. The points $i$ intersects the circles according the decreasing order of the coefficients $\kappa_{k,lm}$. Thus, the order on the circles will be $$C_{k,1,2}, C_{k,1,3}, C_{k,2,3},C_{k,1,4},\dots, C_{k,k-2,k-1}.$$

 Case $t_l<t_k<t_m$.

 Then the circle $C_{klm}$ goes between the chord $P_lP_k$ and the graphics $\Gamma_+$ (see fig.~\ref{fig:circle_cases} middle).  It means the tangent line which lies above $\Gamma_+$, belongs to the third quadrant of the plane.

 Let $r_{klm}$ be the radius of the circle $C_{klm}$ and let $R_p, 3\le p\le n,$ be equal to $\max_{k,l,m\le p}r_{klm}$. Let $\alpha_p, 3\le p\le n,$ be equal to the minimal angle $$\min_{u,v,w\le p, u\ne w}\angle P_uP_vP_w.$$

 \begin{lemma}\label{lem:case2_growth_condition}
 Let $t_1\ge 1$ and $t_i\ge t_{i-1}$ for any $i, 1<i\le n$ and
 \begin{equation}\label{eq:growth_condition_case23}
 t_i\ge\max\left(\frac 3{\sin\alpha_{i-1}}t_{i-1}^2, t_{i-1}^2+2R_{i-1}\right)
 \end{equation}
  for $i>3$. Then for any indices $k,l_1,l_2,m_1,m_2$ such that $1\le l_1<k<m_1\le n$ and $1\le l_2<k<m_2\le n$ we have
\begin{enumerate}
\item if $l_1<k-1$ then the arc $P_{l_1}P_k$ of the circle $C_{kl_1m_1}$ lies in the angle $\angle P_{l_1}P_kP_{l_1+1}$;
\item if $m_1<m_2$ then $r_{kl_1m_1}<r_{kl_2m_2}$;
\item $\kappa_{k,l_1m_1}>\kappa_{k,l_2m_2}$ if and only if $l_1>l_2$ or $l_1=l_2, m_1<m_2$, where $\kappa_{k,l_sm_s}, s=1,2,$ is the slope of the circle $C_{kl_sm_s}$ at the point $P_k$.
\end{enumerate}
 \end{lemma}
\begin{proof}
1. Let $r=r_{kl_1m_1}$ and $\alpha$ be the angle between the circle $C_{kl_1m_1}$ and the chord $P_{l_1}P_k$ at the point $P_k$. Then $\sin\alpha=\frac{P_{l_1}P_k}{2r}$. The growth condition implies
$$2r\ge P_kP_{m_1}>t_{m_1}^2-t_k^2>\left(\frac 3{\sin\alpha_{k}}-1\right)t^2_k>\frac {2t^2_k}{\sin\alpha_{k}}.$$
On the other hand,
$$P_{l_1}P_k=\sqrt{(t_k-t_{l_1})^2+(t_k^2-t_{l_1}^2)^2}<\sqrt{t_k^2+t_k^4}<2t_k^2.$$
Then $\sin\alpha<\sin\alpha_{k}$ and $\alpha<\alpha_{k}\le \angle P_{l_1}P_kP_{l_1+1}$.  Therefore, the arc $P_{l_1}P_k$ lies in the angle $\angle P_{l_1}P_kP_{l_1+1}$.

2. If $m_1<m_2$ then
$$
r_{kl_2m_2}\ge \frac {P_kP_{m_2}}2 > \frac{t_{m_2}^2-t_k^2}2\ge \frac{t_{m_2}^2-t_{m_2-1}^2}2 \ge R_{m_2-1}\ge r_{kl_1m_1}.
$$

3. If $l_1>l_2$ then the arc $P_{l_2}P_k$ of the circle $C_{kl_2m_2}$ lies in the angle $\angle P_{l_2}P_kP_{l_2+1}$ (see Fig.~\ref{fig:case2_ordering} left). Then it lies above the chord $P_{l_2+1}P_k$ and the chord $P_{l_1}P_k$. On the other hand, the arc $P_{l_1}P_k$ of the circle $C_{kl_1m_1}$ lies below the chord $P_{l_1}P_k$. Then the arc $P_{l_2}P_k$ lies above the arc $P_{l_1}P_k$ and $\kappa_{k,l_1m_1}>\kappa_{k,l_2m_2}$.

If $l_1=l_2$ and $m_1<m_2$ then the arc $P_{l_1}P_k$ of the circle $C_{kl_1m_1}$ and the arc $P_{l_2}P_k$ of the circle $C_{kl_2m_2}$ have the same chord but the radius of the circle $C_{kl_1m_1}$ is less than the radius of the circle $C_{kl_2m_2}$ (see Fig.~\ref{fig:case2_ordering} right). Hence, Then the arc $P_{l_2}P_k$ lies above the arc $P_{l_1}P_k$ and $\kappa_{k,l_1m_1}>\kappa_{k,l_2m_2}$.

 \begin{figure}
  \centering
  \begin{tabular}{cc}
    \includegraphics[width=0.25\textwidth]{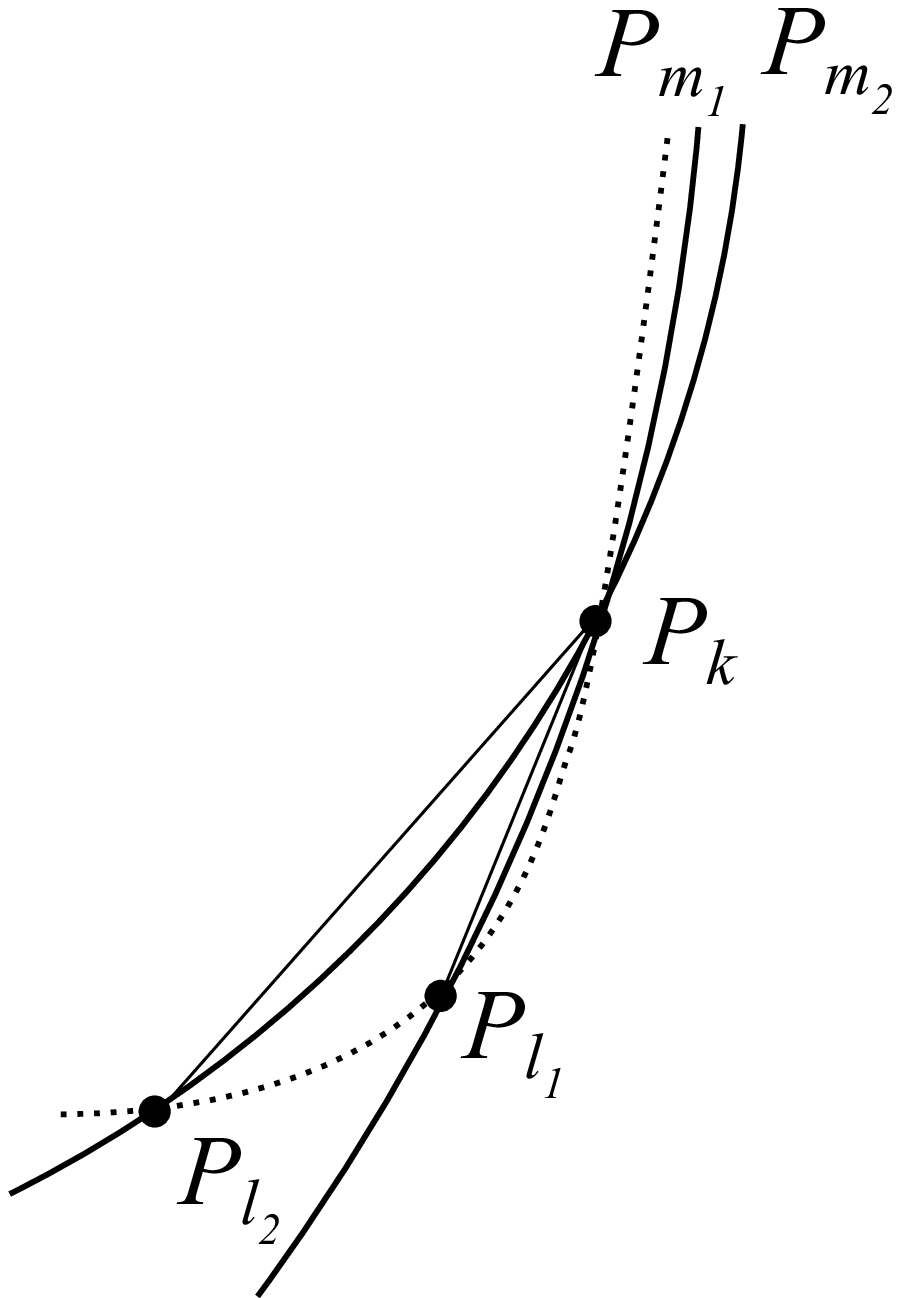} &
    \includegraphics[width=0.25\textwidth]{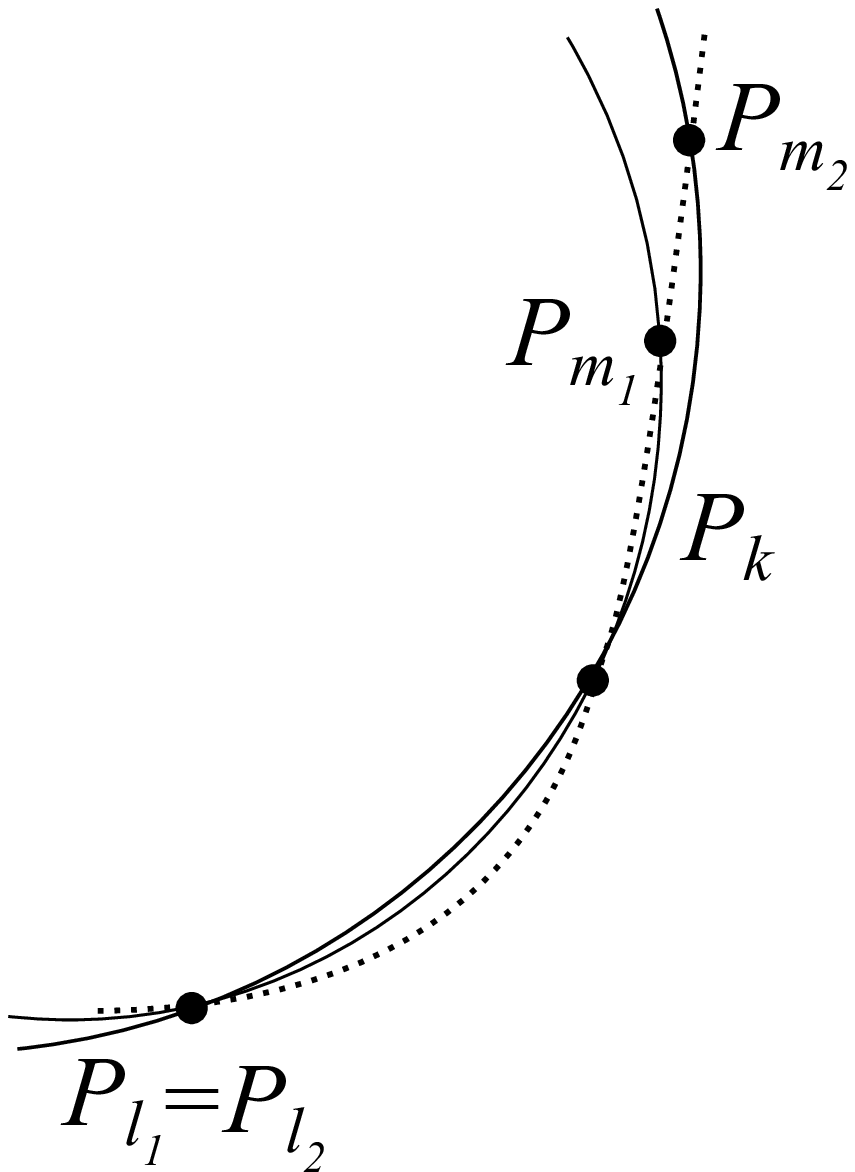}
  \end{tabular}
  \caption{Order of circles in the second case}\label{fig:case2_ordering}
 \end{figure}
\end{proof}

The lemma above establishes the order the point $i$ crosses the circles $C_{klm}$: the circle $C_{k,k-1,k+1}$ goes first, then the circles $C_{k,k-1,k+2}, C_{k,k-1,k+3},\dots, C_{k,k-1,n}$, then $C_{k,k-2,k+1},C_{k,k-2,k+2},\dots, C_{k,k-2,n}$ and so on till $C_{k,1,n}$.

 Case $t_k<t_l<t_m$.

 The circle $C_{klm}$ goes between the chord $P_lP_k$ and the graphics $\Gamma_+$ (see fig.~\ref{fig:circle_cases} right).  It means the tangent line which lies above $\Gamma_+$, belongs to the first quadrant of the plane. The case can be solved in the same manner as the previous case. The order of circles we have here is the following: $$C_{k,n-1,n},C_{k,n-2,n},C_{k,n-2,n-1},C_{k,n-3,n},\dots,C_{k,k+1,k+2}.$$

 Assume the sequence $t_i, 1\le i\le n,$ satisfies the growth conditions~\eqref{eq:growth_condition_case1} and~\eqref{eq:growth_condition_case23}. Then we now the order of circles inside each of three cases. Since the first case includes the circles whose tangent lines lie in the second quadrant, the second case includes circles with tangent lines in the third quadrant and the third case gives circles in the first quadrant, the whole order is the following: the circles of the second case go first, then the circles of the first case, then the circles of the third case. In other words, when the point $i$ rounds the point $P_k$ from above we get the word $c_{ik}$ from~\eqref{eq:d_elements}. Then the dynamical system which correspond to the pure braid $b_{ij}$, leads to the formula~\eqref{eq:maps_to_G4n}.

\end{proof}
}% end of appendix


\begin{thebibliography}{10}

\bibitem{Bard} V.\,G. Bardakov, The Virtual and Universal Braids, {\em Fundamenta Mathematicae}, {\bf 184} (2004), pp.1--18.
 \bibitem{IMN}
 D.\,P. Ilyutko, V.\,O. Manturov, I.\,M. Nikonov, Virtual Knot Invariants Arising
From Parities, {\em Banach Center Publ.}, {\bf 100} (2014), pp.\ 99--130.

\bibitem{M1} V.O. Manturov, Non-reidemeister knot theory and its applications in
dynamical systems, geometry, and topology // arxiv:1501.05208

\bibitem{M2} V.O. Manturov, Parity in knot theory, {\em Sbornik Math.}, {\bf 201}:5 (2010), pp.\ 693--733.

\bibitem{M3} V.O. Manturov, One-Term Parity Bracket For Braids  // arXiv:1501.00580

\end{thebibliography}
\end{document}